\documentclass[11pt]{article}

\usepackage[margin=1in]{geometry}
\usepackage[T1]{fontenc}
\usepackage{lmodern}
\usepackage{microtype}
\usepackage{amsmath,amssymb,amsfonts,amsthm,bm}
\usepackage{mathtools}
\usepackage{physics}
\usepackage{booktabs}
\usepackage{xcolor}
\usepackage{graphicx}
\usepackage{hyperref}
\usepackage[capitalize,nameinlink]{cleveref}
\usepackage{authblk}
\usepackage{enumitem}

\usepackage{todonotes}

\newtheorem{theorem}{Theorem}
\newtheorem{lemma}[theorem]{Lemma}

\newtheorem{proposition}[theorem]{Proposition}
\newtheorem{remark}[theorem]{Remark}

\title{\textbf{Fisher-Information-Based Sensor Placement for \\ Structural Digital Twins: \\ Analytic Results and Benchmarks}
\thanks{This work is partially supported by the Office of Naval Research (ONR) under Award NO: N00014-24-
1-2147, NSF grant DMS-2408877, and the Air Force Office of Scientific Research (AFOSR) under Award NO: FA9550-25-1-0231.}
}
\date{}

	\author[1]{Harbir Antil\thanks{Email: \texttt{hantil@gmu.edu}}}	
    \author[2]{Animesh Jain\thanks{Email: \texttt{ajain34@gmu.edu}}}	
	\author[3]{Rainald L\"ohner\thanks{Email: \texttt{rlohner@gmu.edu}}}    
	\affil[1,2]{\small{Department of Mathematical Sciences and the Center for Mathematics and Artificial Intelligence (CMAI), George Mason University, Fairfax, VA 22030, USA.}}
	\affil[3]{\small{Department of Physics and the Computational Fluid Dynamics Center, George Mason University, Fairfax, VA 22030, USA.}}

\newcommand{\R}{\mathbb{R}}

\newcommand{\T}{\mathsf{T}}

\DeclareMathOperator*{\argmin}{arg\,min}

\begin{document}
\maketitle

\begin{abstract}
High-fidelity digital twins rely on the accurate assimilation of sensor data into
physics-based computational models.
In structural applications, such twins aim to identify spatially distributed
quantities—such as elementwise weakening fields, material parameters, or
effective thermal loads—by minimizing discrepancies between measured and
simulated responses subject to the governing equations of structural mechanics.
While adjoint-based methods enable efficient gradient computation for these
inverse problems, the quality and stability of the resulting estimates depend
critically on the choice of sensor locations, measurement types, and directions.

This paper develops a rigorous and implementation-ready framework for
Fisher-information-based sensor placement in adjoint-based finite-element
digital twins.
Sensor configurations are evaluated using a D-optimal design criterion derived
from a linearization of the measurement map, yielding a statistically meaningful
measure of information content.
We present matrix-free operator formulas for applying the Jacobian and its
adjoint, and hence for computing Fisher-information products
$Fv = J^\top R^{-1} Jv$ using only forward and adjoint solves.
Building on these operator evaluations, we derive explicit sensitivity
expressions for D-optimal sensor design with respect to measurement parameters
and discuss practical strategies for evaluating the associated
log-determinant objectives.
To complement the general framework, we provide analytically tractable sensor
placement results for a canonical one-dimensional structural model, clarifying
the distinction between detectability and localizability and proving that
D-optimal placement of multiple displacement sensors yields approximately
uniform spacing.
\end{abstract}

\section{Introduction}

Throughout the life cycle of engineering structures, material aging,
fatigue, and damage progressively degrade mechanical properties.
Digital twins—virtual replicas synchronized with measurements from the
physical system—enable continuous monitoring, system identification,
and predictive maintenance.
For structural systems, the identification task can be formulated as an
inverse problem in which one minimizes a misfit between computed and
measured quantities, such as displacements or strains, subject to the
partial differential equations governing structural mechanics.
Adjoint-based techniques make it possible to compute gradients of such
misfit functionals with respect to high-dimensional unknowns—such as
elementwise weakening fields $\alpha$, material parameters $\beta$, or
effective thermal loading parameters $f_{\Delta T}$—using only a small
number of forward and adjoint solves
\cite{FAiraudo_RLoehner_HAntil_2023a,TAnsari_RLohner_RWuchner_HAntil_2025a,
TAnsari_RLoehner_RWuechner_HAntil_SWarnakulasuriya_IAntonau_FAiraudo_2025a,
IAntonau_SWarnakulasuriya_TAnsari_RWuchner_RLohner_HAntil_2025,
IAntonau_SWarnakulasuriya_RWuechner_FAiraudo_RLoehner_HAntil_TAnsari_2025a,
RLoehner_Fairaudo_HAntil_RWuechner_SWarnakulasuriya_IAntonau_TAnsari_2025a}.

A persistent challenge for practical deployment of digital twins is the
limited number of sensors available for measurement.
Sparse or poorly placed sensors render the inverse problem severely
ill-posed and sensitive to noise, while redundant sensors increase cost
and complexity without commensurate gains in information content.
Consequently, the design of the sensor configuration itself—namely,
\emph{which locations, quantities, and directions should be measured}—is
a critical factor governing identifiability, stability, and predictive
capability.

This work develops a mathematically precise and implementation-ready
framework for sensor placement in adjoint-based finite-element digital
twins based on Fisher information.
Rather than attempting to solve a fully coupled bilevel design–inverse
optimization problem, sensor configurations are evaluated through a
Fisher-information-based criterion obtained by linearizing the
measurement map at a reference parameter.
Within this framework, Jacobian and adjoint-Jacobian operator actions,
and hence Fisher-information products of the form
$Fv = J^\top R^{-1}Jv$, can be computed in a matrix-free manner using only
forward and adjoint solves, while design sensitivities are expressed
explicitly in terms of the associated measurement operators.

\medskip
\noindent
\textbf{Contributions.}
The main contributions of this paper are:
\begin{itemize}[leftmargin=2em]
\item[(i)]
A Fisher-information-based formulation of sensor placement tailored to
adjoint-based finite-element digital twins for structural mechanics,
with a precise characterization of how sensor configurations enter
through discrete measurement operators.
\item[(ii)]
Matrix-free operator expressions for Jacobian and adjoint-Jacobian actions,
and hence for Fisher-information products of the form
$Fv = J^\top R^{-1} J v$, enabling information-based sensor evaluation
using only forward and adjoint solves consistent with existing inverse
problem solvers.
\item[(iii)]
A sensor-operator-centric sensitivity analysis showing that derivatives
of D-optimal design objectives with respect to sensor design parameters
depend only on variations of the measurement operators, and not on
second derivatives of the forward PDE solution.
\item[(iv)]
Analytically tractable benchmark results for a canonical one-dimensional
structural model, including a rigorous distinction between
\emph{detectability} and \emph{localizability}, and a closed-form proof
that D-optimal placement of multiple displacement sensors yields
(approximately) uniformly spaced sensor locations.
\end{itemize}

\paragraph{State of the art and related work.}
Sensor placement for structural systems has a long history in structural
dynamics, model updating, and structural health monitoring (SHM).
Early influential work proposed modal-based placement criteria, including
effective-independence (EI) strategies that iteratively select sensor locations
to improve the conditioning and information content of a targeted modal subspace.
These EI procedures are commonly expressed in terms of an information matrix and
its determinant/trace, and are closely aligned in structure with D- and A-optimal
design objectives for modal identification \cite{DCKammer_1991a}.

Information-theoretic formulations were subsequently developed for parametric
identification of structural models under noise.
Papadimitriou \cite{CPapadimitriou_2004a} proposed sensor placement based on an
information-entropy measure of parameter uncertainty and derived an asymptotic
approximation (large-data regime) showing that sensor configurations can be
ranked using only a nominal structural model, without requiring access to the
full time-history data at the design stage. The resulting design problems are
discrete and are typically addressed using sequential sensor placement
heuristics or combinatorial/metaheuristic search.

In parallel, optimal experimental design (OED) for large-scale inverse problems
governed by PDEs has developed scalable \emph{matrix-free} methodologies, since
explicit formation of Fisher, Hessian, or covariance matrices is often
infeasible. For a first monograph on OED that specifically addressed PDE-based inverse problems, we refer to Uci\'nski \cite{DUcinski_2005a}.  
Haber, Horesh, and Tenorio \cite{EHaber_LHoresh_LTenorio_2008a} developed design
criteria for linear ill-posed inverse problems that explicitly quantify
bias--variance tradeoffs of Tikhonov-regularized estimators via mean-square error
(MSE), and they evaluate the resulting objectives/gradients through operator
actions and stochastic trace estimators rather than dense matrix algebra.
A comprehensive account of Bayesian OED for PDE-constrained inverse problems,
including A-/D-optimality, expected information gain, infinite-dimensional
formulations, and randomized trace/log-determinant estimators, is given in the
topical review \cite{AAlexanderian_2021a}. See also \cite{AAlexanderian_NPetra_GStadler_OGhattas_2016a,WWu_PChen_OGhattas_2023a} for the article on related Bayesian inverse problems and ODE using adjoint approaches. Finally, we mention a comprehensive review article on OED \cite{XHuan_JJagalur_YMarzouk_2024a} and inverse problems \cite{Ghattas_Willcox_2021a}.

\paragraph{Positioning of this work.}
In contrast to the above OED literature---which is often presented at the level
of abstract parameter-to-observable maps and posterior covariance operators---we
focus on adjoint-based \emph{finite-element digital twin} workflows in structural
mechanics, where sensor configurations enter explicitly through discrete
measurement operators acting on the FE state (e.g., nodal displacement
interpolation and strain recovery at sensor locations).
Within a Fisher-information (local linearization) design criterion, we show how
Jacobian and adjoint-Jacobian \emph{operator actions}, and hence Fisher products
$Fv=J^\top R^{-1}Jv$, can be computed using forward/adjoint solves consistent
with existing inverse solvers.
We further provide analytically tractable benchmark results for a canonical 1D
bar model, including a rigorous detectability--localizability distinction and a
closed-form characterization of D-optimal placement of multiple displacement
sensors; these proofs complement the general framework and offer insight beyond
purely numerical design studies.

\section{Forward and Inverse Problem Formulation}

We consider a structure discretized by finite elements under $n$ prescribed
load cases.  
Let $u_i \in \R^{N_u}$ denote the vector of nodal displacements for the
$i$-th load case and $f_i$ the corresponding external force vector.  
The structural equilibrium equations read
\begin{equation}
  K(\alpha,\beta)\,u_i = f_i + f_{\Delta T}, \qquad i = 1,\dots,n,
  \label{eq:FE}
\end{equation}
where the stiffness matrix is assembled elementwise as
\begin{equation}
  K(\alpha,\beta)
  = \sum_{e=1}^{N_e} \alpha_e\, K_e(\beta),
  \label{eq:stiffness}
\end{equation}
and $\alpha_e \in [\alpha_{\min},1], \quad \alpha_{\min}>0$, represents the elementwise strength factor capturing
damage or weakening, while $\beta$ denotes material parameters such as the
Young’s modulus or Poisson’s ratio.
The term $f_{\Delta T}$ accounts for internal forces due to thermal effects. \\ 
We assume that the finite element discretization incorporates sufficient essential boundary conditions so that, for all admissible parameters $(\alpha,\beta)$, the stiffness matrix $K(\alpha,\beta)$ is symmetric positive definite. 
The element matrices $K_e(\beta)$ are assumed to depend smoothly on the material parameters $\beta$, ensuring Fréchet differentiability of the solution map $(\alpha,\beta,f_{\Delta T})\mapsto u_i$.\\
The parameters $\alpha$ and $\beta$ are assumed to belong to admissible sets $\mathcal{A}\subset\mathbb{R}^{N_e}$ and $\mathcal{B}\subset\mathbb{R}^{N_\beta}$, respectively, which are closed and convex.

\medskip
\noindent
The system response is compared to measured data at $m$ sensor locations
$x_j$, $j=1,\dots,m$.  
For each load case $i$, displacement and strain measurements are modeled via
interpolation operators $I^d_{ij}$ and $I^s_{ij}$ acting on the finite element
solution:
\begin{align*}
  u^{\mathrm{md}}_{ij} &\approx I^d_{ij}\,u_i,\\
  s^{\mathrm{ms}}_{ij} &\approx I^s_{ij}\,s_i(u_i),
\end{align*}
where $s_i(u_i)$ denotes the finite-element strain field derived from the
displacements. The operators $I^d_{ij}$ and $I^s_{ij}$ are assumed to be bounded linear maps,
corresponding to interpolation or localized averaging of the finite element
solution and strain field at sensor locations. 

\medskip
\noindent
Given measured data $\{u^{\mathrm{md}}_{ij}, s^{\mathrm{ms}}_{ij}\}$, the
parameter identification problem seeks $(\alpha,\beta,f_{\Delta T})$ that minimize
the weighted least-squares misfit
\begin{equation}
\label{eq:cost}
  I(u_1,\ldots,u_n,\alpha,\beta,f_{\Delta T})
  = \frac12 \sum_{i=1}^{n}\sum_{j=1}^{m}
    w^{\mathrm{md}}_{ij}\,\|u^{\mathrm{md}}_{ij} - I^d_{ij}u_i\|^2
    + \frac12 \sum_{i=1}^{n}\sum_{j=1}^{m}
    w^{\mathrm{ms}}_{ij}\,\|s^{\mathrm{ms}}_{ij} - I^s_{ij}s_i(u_i)\|^2.
\end{equation}
Here $\|\cdot\|$ denotes the Euclidean norm in the corresponding measurement space.
The weights
$w^{\mathrm{md}}_{ij}$ and $w^{\mathrm{ms}}_{ij}$ are strictly positive and
reflect measurement noise levels, yielding a statistically meaningful
least-squares formulation.
The problem is constrained by~\eqref{eq:FE}--\eqref{eq:stiffness}.
The corresponding Lagrangian reads
\begin{equation}
  \mathcal{L}(u_i,\tilde u_i,\alpha,\beta,f_{\Delta T})
  = I(u_i,\alpha,\beta,f_{\Delta T})
  + \sum_{i=1}^n \tilde u_i^\T \big( K(\alpha,\beta)u_i - f_i - f_{\Delta T}\big),
  \label{eq:lagrangian}
\end{equation}
where $\tilde u_i$ are the adjoint variables.
Variations of $\mathcal{L}$ yield the standard adjoint system, gradient
expressions, and update formulas as detailed in~\cite{FAiraudo_RLoehner_HAntil_2023a,TAnsari_RLoehner_RWuechner_HAntil_SWarnakulasuriya_IAntonau_FAiraudo_2025a}.

\section{Sensor Operator}

Let $S$ denote the \emph{sensor configuration}, encompassing all design
choices, including:
\begin{itemize}[leftmargin=2em]
  \item spatial positions $x_j$,
  \item measurement types (displacement, strain, acceleration, etc.),
  \item measurement directions or components, and
  \item weighting factors reflecting reliability or cost.
\end{itemize}
We assume that $S$ belongs to a prescribed admissible set $\mathcal{S}_{\rm ad}$ encoding
geometric, physical, and budgetary constraints.

\medskip
\noindent
The collection of interpolation operators $\{I^d_{ij},I^s_{ij}\}$ associated
with $S$ is represented by a single measurement operator
$M_i(S) \in \mathbb{R}^{m_i \times N_u}$ (here $m_i$ is the number of measurements for $i$-th load) acting linearly on the discrete
displacement field. The operator $M_i(S)$ incorporates bounded linear finite
element mappings (e.g.\ strain–displacement operators) composed with sensor
selection and weighting. The resulting observation model reads
\begin{equation}
  y_i(S,q) = M_i(S)\,u_i(q),
  \qquad
  y_i^{\mathrm{obs}} = y_i(S,q^\star) + \eta_i,
  \label{eq:measure}
\end{equation}
where $q = (\alpha,\beta,f_{\Delta T})$ denotes the unknown parameters,
$q^\star$ is the (unknown) true parameter vector, and $\eta_i$ denotes
measurement noise.

\medskip
\noindent
For a fixed sensor configuration $S$, the inverse problem minimizes the
data misfit
\begin{equation}
  J(q;S) =
  \frac12 \sum_{i=1}^{n}
  \|W_i\big(y_i^{\mathrm{obs}} - y_i(S,q)\big)\|^2,
  \label{eq:misfit}
\end{equation}
subject to~\eqref{eq:FE}. Here $W_i$ are symmetric positive definite weighting
matrices.

\section{Optimization Formulation and Fisher Information}

We denote by $q \in \mathbb{R}^{N_q}$ the vector of parameters to be identified,
for example
\[
  q = (\alpha,\beta,f_{\Delta T}),
\]
or a subset thereof depending on the application.
Let $q^\star$ denote the (unknown) true parameter and $q_0$ a reference
parameter about which we linearize the measurement map.
Typical choices for $q_0$ include:
\begin{itemize}[leftmargin=2em]
  \item a nominal or prior model (e.g.\ undamaged structure), or
  \item the current iterate of an inverse solver in an outer optimization loop.
\end{itemize}

\subsection{Conceptual bilevel structure}

At a conceptual level, optimal sensor placement for digital twins is naturally
formulated as a bilevel optimization problem.  
The \emph{inner} (lower-level) problem is the parameter identification task:
for a fixed sensor configuration $S$ and measured data $y^{\mathrm{obs}}$, one
computes an estimate $\hat q(S)$ by solving
\begin{equation}
  \hat q(S)
  \in \argmin_{q \in \mathbb{R}^{N_q}}
    \Big\{ J(q;S) + R(q) \Big\},
  \label{eq:lower-level}
\end{equation}
where $J(q;S)$ is the misfit \eqref{eq:misfit} and $R_q(q)$ denotes
regularization terms encoding prior information or smoothness of the parameters.
In practice, \eqref{eq:lower-level} is solved by an adjoint-based PDE-constrained
optimization algorithm \cite{HAntil_DPKouri_MDLacasse_DRidzal_2018a}.

The \emph{outer} (upper-level) problem aims to choose the sensor configuration
$S$ so that the estimator $\hat q(S)$ is as accurate and robust as possible.  
A generic formulation can be written as
\begin{equation}
  \min_{S \in \mathcal{S}_{\mathrm{ad}}}
    \Phi(S)
  \quad\text{with}\quad
  \Phi(S) := \mathbb{E}\big[
    \mathrm{loss}\big(q^\star,\hat q(S)\big)
  \big],
  \label{eq:upper-level}
\end{equation}
where $\mathcal{S}_{\mathrm{ad}}$ denotes the set of admissible sensor
configurations, $\mathrm{loss}(\cdot,\cdot)$ quantifies the quality of the
estimate (e.g.\ squared error, risk measures), and the expectation may be taken
over measurement noise, loading scenarios, and prior uncertainty in $q^\star$.

The combination of \eqref{eq:lower-level} and \eqref{eq:upper-level} yields a
bilevel problem: the outer design variable $S$ is optimized based on the
solution $\hat q(S)$ of an inner PDE-constrained optimization problem.
Solving this bilevel problem directly is prohibitively expensive.  
Instead, we approximate the statistical behavior of $\hat q(S)$ using the
local Fisher information and obtain a tractable single-level design problem in
terms of $F(S)$; this is the topic of the next subsections.

\begin{remark}[Offline design versus online inference]
\rm 
In a typical digital twin workflow, sensor placement is performed in an
\emph{offline} design phase, before or in parallel with the acquisition of
operational data.
Once a sensor layout $S^\star$ has been selected, the digital twin operates in
an \emph{online} phase in which inverse problems of the form
\eqref{eq:lower-level} are repeatedly solved for different measurements using
the fixed configuration $S^\star$.
The Fisher-information-based criteria discussed below provide an offline
surrogate for the expected performance of these online inversions.
\end{remark}

\subsection{Linearized statistical model and misfit}

For a given sensor configuration $S$, we stack all measurements over all load
cases (and, for transient problems, over all time steps) into a single vector
\[
  y(S,q) \in \mathbb{R}^{N_y}, \quad \mbox{where} \quad 
  \sum_{i=1}^n m_i = N_y ,
\]
where $N_y$ is the total number of observations. 
We assume the additive noise model
\begin{equation}
  y^{\mathrm{obs}} = y(S,q^\star) + \eta,
  \qquad
  \eta \sim \mathcal{N}(0,R),
  \label{eq:noise-model}
\end{equation}
where $R \in \mathbb{R}^{N_y \times N_y}$ is the covariance matrix of the
measurement noise.
In many applications $R$ is taken to be diagonal, with entries corresponding to
the variances of individual sensors.

In the least-squares formulation \eqref{eq:misfit} we wrote
\[
  J(q;S) = \frac12 \sum_{i=1}^{n}
  \big\| W_i\big(y_i^{\mathrm{obs}} - y_i(S,q)\big)\big\|^2,
\]
where $W_i$ are weighting matrices for each load case.
If we stack the $y_i$ into a global data vector $y(S,q)$ and define a
corresponding block-diagonal weighting matrix $W$, we obtain
\[
  J(q;S) = \frac12 \big\| W\big(y^{\mathrm{obs}} - y(S,q)\big)\big\|^2.
\]
Under the Gaussian noise model \eqref{eq:noise-model}, the (negative) log-likelihood
(up to an additive constant independent of $q$) is
\[
  \mathcal{L}(q;S)
  := \frac12 \big( y^{\mathrm{obs}} - y(S,q) \big)^\top R^{-1}
                   \big( y^{\mathrm{obs}} - y(S,q) \big),
\]
while the log-likelihood is
\[
  \ell(q; y^{\mathrm{obs}}, S) := \log p(y^{\mathrm{obs}}\mid q,S)
  = -\,\mathcal{L}(q;S) - \tfrac12 \log\det(2\pi R).
\]
Thus choosing $W$ with $W^\top W \approx R^{-1}$ makes the least-squares misfit
\eqref{eq:misfit} consistent with this statistical model.

To make contact with classical optimal experimental design, we linearize
$y(S,q)$ around the reference parameter $q_0$:
\begin{equation}
  y(S,q) \approx y(S,q_0) + J(S)\,(q - q_0),
  \qquad
  J(S) := \frac{\partial y(S,q)}{\partial q}\bigg|_{q=q_0}
  \in \mathbb{R}^{N_y \times N_q}.
  \label{eq:linearization}
\end{equation}
If $y(S,q)$ depends linearly on $q$, this representation is exact globally; for
nonlinear models it provides a local approximation around $q_0$.
The sensitivities encoded in $J(S)$ are computed efficiently by adjoint-based
methods.

\subsection{Fisher information in the linear--Gaussian setting}

The Fisher information at $q_0$ is
\[
  F(q_0,S)
  = \mathbb{E}\big[ s(q_0; y, S)\, s(q_0; y, S)^\top \big],
  \qquad s(q; y, S) := \nabla_q \ell(q; y, S),
\]
where the expectation is taken with respect to the distribution of the random
observation $y$ at $q=q_0$.
Under standard regularity assumptions (e.g.\ dominated convergence and
sufficient smoothness of $\ell$), this is equivalently
\[
  F(q_0,S)
  = -\,\mathbb{E}\!\left[ \nabla_q^2 \ell(q; y, S) \right]_{q=q_0}
  = \mathbb{E}\!\left[ \nabla_q^2 \mathcal{L}(q; S) \right]_{q=q_0},
\]
which is often more convenient for computation.

Using the linearized measurement model \eqref{eq:linearization} and the
Gaussian noise model \eqref{eq:noise-model}, the log-likelihood (up to an
additive constant) can be written as
\[
  \ell(q; y^{\mathrm{obs}}, S)
  \approx -\frac{1}{2}
  \big( y^{\mathrm{obs}} - y(S,q_0) - J(S)(q-q_0) \big)^\top
  R^{-1}
  \big( y^{\mathrm{obs}} - y(S,q_0) - J(S)(q-q_0) \big).
\]
Differentiating with respect to $q$ we obtain the score
\[
  \nabla_q \ell(q; y^{\mathrm{obs}}, S)
  \approx J(S)^\top R^{-1}
  \big( y^{\mathrm{obs}} - y(S,q_0) - J(S)(q-q_0) \big),
\]
and differentiating once more yields the Hessian
\[
  \nabla_q^2 \ell(q; y^{\mathrm{obs}}, S)
  \approx -\,J(S)^\top R^{-1} J(S),
\]
which is independent of the particular realization $y^{\mathrm{obs}}$ in the
linearized model.
Therefore, the Fisher information at $q_0$ is given by
\begin{equation}
  F(S) := F(q_0,S)
  = -\mathbb{E}\!\left[ \nabla_q^2 \log p(y^{\mathrm{obs}} \mid q,S) \right]_{q=q_0}
  = J(S)^\top R^{-1} J(S).
  \label{eq:fisher-final}
\end{equation}
In the exactly linear--Gaussian case this identity is exact; for nonlinear
forward models it represents the local Fisher information obtained from the
linearization \eqref{eq:linearization} and coincides with the Gauss--Newton
approximation to the Hessian of the negative log-likelihood.

Intuitively, the $k$-th column of $J(S)$ measures how the data change in
response to a unit perturbation of the $k$-th parameter.
Pre- and post-multiplication by $R^{-1}$ weights these changes by the
measurement noise, and the quadratic form
\[
  v^\top F(S) v = \big\| R^{-1/2} J(S) v \big\|^2
\]
quantifies how distinguishable a parameter perturbation $v$ is in the presence
of noise.

\subsection{Sensor design problem}

Recall that the Fisher information matrix $F(S)$ in \eqref{eq:fisher-final}
depends on the sensor configuration $S$ and on the chosen linearization point
$q_0$.
To keep the notation light we suppress this dependence and simply write
$F(S)$, with the understanding that all quantities are evaluated at a fixed
reference parameter $q_0$ (e.g.\ the nominal undamaged structure).

Recall $\mathcal{S}_{\mathrm{ad}}$ denote the set of admissible sensor
configurations, encoding constraints on the number, location, type, and cost of
sensors.
The sensor design problem is then posed as the optimization of a scalar
criterion of $F(S)$ over $S \in \mathcal{S}_{\mathrm{ad}}$.
Classical criteria from optimal experimental design include:
\begin{itemize}[leftmargin=2em]
  \item \textbf{A-optimality:} minimize the average variance,
  \[
    \min_{S \in \mathcal{S}_{\mathrm{ad}}}
      \operatorname{tr}\big( F(S)^{-1} \big);
  \]
  \item \textbf{D-optimality:} maximize the information volume,
  \[
    \max_{S \in \mathcal{S}_{\mathrm{ad}}}
      \log\det F(S);
  \]
  \item \textbf{E-optimality:} maximize the smallest eigenvalue,
  \[
    \max_{S \in \mathcal{S}_{\mathrm{ad}}}
      \lambda_{\min}\big(F(S)\big).
  \]
\end{itemize}
These objectives are well defined when $F(S)$ is symmetric positive definite,
which corresponds to local identifiability of the parameters at $q_0$ under the
sensor configuration $S$.
In practice one may add a small regularization $F(S) \mapsto F(S)+\varepsilon I$
to handle near-singular cases.

In this work we will mostly focus on the D-optimal design, which seeks to
maximize the determinant of $F(S)$, or equivalently to minimize
$-\log\det F(S)$:
\begin{equation}
  \min_{S \in \mathcal{S}_{\mathrm{ad}}}
    \Phi_{\mathrm{D}}(S),
  \qquad
  \Phi_{\mathrm{D}}(S) := -\log\det F(S).
  \label{eq:D-optimal}
\end{equation}

More generally, one may consider robust or scenario-based designs by replacing
$F(S)$ in \eqref{eq:D-optimal} with an average of Fisher information matrices
computed at several representative parameter values $q_0^{(1)},\dots,q_0^{(L)}$
(e.g.\ corresponding to different damage patterns or operating conditions),
\[
  \bar F(S) := \sum_{\ell=1}^L w_\ell F\big(S; q_0^{(\ell)}\big),
\]
and then optimizing a scalar criterion of $\bar F(S)$.
This allows the sensor layout to perform well across a range of plausible
scenarios without requiring knowledge of the true parameter $q^\star$.

In subsequent sections we will derive 
expressions for the
sensitivities of $\Phi_{\mathrm{D}}(S)$ (and related criteria) with respect to
the sensor configuration $S$, and present practical algorithms for computing
near-optimal sensor layouts in the finite element setting.

\section{Adjoint-Based Sensitivities for Sensor Placement}

In this section we derive the expressions needed to compute the gradient of the
D-optimal objective
\begin{equation}
  \Phi_{\mathrm{D}}(S)
  = -\log\det F(S),
  \qquad
  F(S) = J(S)^\top R^{-1} J(S),
  \label{eq:D-objective}
\end{equation}
with respect to the design variables defining the sensor configuration~$S$.
The dependence of $J(S)$ on~$S$ enters through the measurement operators
$M_i(S)$ that extract or interpolate state quantities at the sensor locations.

\subsection{Variation of the D-optimal objective}

We recall the D-optimal sensor design objective
\begin{equation}
  \Phi_{\mathrm{D}}(S)
  = -\log\det F(S),
  \qquad
  F(S) = J(S)^\top R^{-1} J(S),
  \label{eq:D-objective-again}
\end{equation}
where $F(S)$ is symmetric positive definite for admissible sensor
configurations $S$.
The first variation of $\Phi_{\mathrm{D}}$ with respect to a perturbation
$\mathrm{d}F(S)$ is obtained from the standard identity
$\mathrm{d}\log\det F = \operatorname{tr}(F^{-1}\mathrm{d}F)$:
\begin{equation}
  \mathrm{d}\Phi_{\mathrm{D}}
  = -\operatorname{tr}\big( F(S)^{-1} \,\mathrm{d}F(S) \big).
  \label{eq:dPhiF}
\end{equation}

Since $F(S) = J(S)^\top R^{-1} J(S)$ and $R$ does not depend on $S$, the
variation of $F$ induced by a perturbation $\mathrm{d}J(S)$ is
\begin{equation}
  \mathrm{d}F(S)
  = \mathrm{d}\big(J^\top R^{-1} J\big)
  = (\mathrm{d}J)^\top R^{-1} J
    + J^\top R^{-1} (\mathrm{d}J),
  \label{eq:dF}
\end{equation}
where we use the shorthand $J = J(S)$.
Substituting \eqref{eq:dF} into \eqref{eq:dPhiF} yields
\begin{align*}
  \mathrm{d}\Phi_{\mathrm{D}}
  &= -\operatorname{tr}\!\big( F^{-1} (\mathrm{d}J)^\top R^{-1} J \big)
     -\operatorname{tr}\!\big( F^{-1} J^\top R^{-1} (\mathrm{d}J) \big).
\end{align*}
Using the facts that $\operatorname{tr}(X) = \operatorname{tr}(X^\top)$ and
$\operatorname{tr}(ABC) = \operatorname{tr}(BCA)$ (cyclic invariance of the
trace), we first transpose the first term:
\[
  \operatorname{tr}\!\big( F^{-1} (\mathrm{d}J)^\top R^{-1} J \big)
  = \operatorname{tr}\!\big( J^\top R^{-1} (\mathrm{d}J) F^{-1} \big),
\]
and then cyclically permute the right-hand side to obtain
\[
  \operatorname{tr}\!\big( J^\top R^{-1} (\mathrm{d}J) F^{-1} \big)
  = \operatorname{tr}\!\big( F^{-1} J^\top R^{-1} (\mathrm{d}J)  \big).
\]
Hence both contributions are identical, and we obtain
\begin{equation}
  \mathrm{d}\Phi_{\mathrm{D}}
  = -2\,\operatorname{tr}\!\big(
      F(S)^{-1} J(S)^\top R^{-1}\, \mathrm{d}J(S)
    \big).
  \label{eq:dPhi_dJ-correct}
\end{equation}
It is convenient to introduce the ``adjoint weight'' matrix
\begin{equation}
  B(S) := R^{-1} J(S) F(S)^{-1} \in \mathbb{R}^{N_y \times N_q},
  \label{eq:B-def}
\end{equation}
for which $B(S)^\top = F(S)^{-1} J(S)^\top R^{-1}$.
Then \eqref{eq:dPhi_dJ-correct} can be written more compactly as
\begin{equation}
  \mathrm{d}\Phi_{\mathrm{D}}
  = -2\,\operatorname{tr}\!\big( B(S)^\top\, \mathrm{d}J(S) \big).
  \label{eq:dPhi_BdJ}
\end{equation}
Thus, the sensitivity of the D-optimal objective with respect to any design
parameter enters only through the directional derivative $\mathrm{d}J(S)$ of
the sensitivity matrix $J(S)$, and no second derivatives of the forward map
with respect to the parameters $q$ are required.
In the following we express $\mathrm{d}J(S)$ in terms of forward and adjoint
solutions of the governing finite element equations.

\subsection{Structure of the Jacobian $J(S)$}

For each load case $i=1,\dots,n$, the measurement vector is
\[
  y_i(S,q) = M_i(S)\,u_i(q),
\]
where $M_i(S)$ is the measurement (interpolation, extraction, or averaging)
operator determined by the sensor configuration $S$. 
The state $u_i(q)\in\mathbb{R}^{N_u}$ solves the static equilibrium equations
\begin{equation}
  K(\alpha,\beta)\,u_i(q)
  = f_i(q) + f_{\Delta T}(q),
  \label{eq:FE-state}
\end{equation}
with $q$ collecting all inversion parameters (e.g.\ spatially varying material
weakening parameters, thermal loading parameters, global scalars, etc.).

For Fisher information we require the Jacobian of the \emph{measurement map}
with respect to the parameters $q$, evaluated at a reference $q_0$:
\[
  J(S)
  :=
  \frac{\partial y(S,q)}{\partial q}\Big|_{q=q_0}
  \;\in\; \mathbb{R}^{N_y\times N_q},
\]
where $y(S,q)$ stacks the $y_i(S,q)$ over all load cases.
Since the sensor configuration $S$ is fixed here, the measurement operators
$M_i(S)$ do not depend on $q$.

A parameter perturbation $\delta q$ induces a state perturbation
$\delta u_i\in\mathbb{R}^{N_u}$ satisfying the linearized equilibrium
equation
\begin{equation}
  K(\alpha,\beta)\,\delta u_i
  =
  \Big(       
      \frac{\partial f_{\Delta T}}{\partial q}
      - \frac{\partial K(\alpha,\beta)}{\partial q}\,u_i
  \Big)\,\delta q
  \;:=\;C_i\,\delta q,
  \label{eq:linstate}
\end{equation}
where $C_i\in\mathbb{R}^{N_u\times N_q}$ is the parameter-to-RHS sensitivity
operator.
Thus
\[
  \delta y_i
  = M_i(S)\,\delta u_i
  = M_i(S)\,U_i\,\delta q,
\]
where $U_i$ is the matrix of state sensitivities,
\[
  U_i := \frac{\partial u_i(q)}{\partial q}\Big|_{q=q_0}
        \in\mathbb{R}^{N_u\times N_q},
\quad
  K(\alpha,\beta)\,U_i = C_i.
\]
Stacking all load cases yields the block Jacobian
\[
  J(S)
  =
  \begin{bmatrix}
    M_1(S)\,U_1 \\
    \vdots \\
    M_n(S)\,U_n
  \end{bmatrix}
  \in\mathbb{R}^{N_y\times N_q}.
\]

The $k$-th column of $J(S)$ corresponds to the response of the full sensor
dataset to a unit perturbation of the $k$-th parameter component $q_k$ and is
obtained by solving~\eqref{eq:linstate} with $\delta q=e_k$.

Direct formation of $J(S)$ is practical when $N_q$ is small (e.g.\ global
parameters).  
When $q$ is high-dimensional (e.g.\ spatially distributed weakening fields),
explicit construction of $U_i$ and $J(S)$ becomes prohibitive.
In such cases we rely on the \emph{operator view} of $J$, applying $J$ and
$J^\top$ without forming them explicitly.

\paragraph{Adjoint/operator view (matrix-free).}

The Jacobian $J(S)$ is the derivative of the stacked measurement vector
$y(S,q)$ with respect to the parameter vector $q$, evaluated at $q_0$.
Given a direction $v\in\mathbb{R}^{N_q}$, the product $Jv$ is the
directional derivative of $y$ at $q_0$ in direction $v$:
\[
  Jv
  = \left.\frac{\mathrm{d}}{\mathrm{d}\varepsilon}
           y(S,q_0 + \varepsilon v)\right|_{\varepsilon=0}.
\]
We now show how to compute $Jv$, $J^\top w$, and $Fz:=J^\top R^{-1}Jz$
without forming $J$ explicitly.

\medskip\noindent
\textbf{1. Forward (linearized) solve for $Jv$.}
Fix $S$ and a direction $v\in\mathbb{R}^{N_q}$.
For each load case $i$, the state $u_i(q)$ solves
\[
  K(\alpha,\beta)\,u_i(q)
  = f_i(q) + f_{\Delta T}(q).
\]
Perturb $q$ to $q_0 + \varepsilon v$ and differentiate at $\varepsilon=0$.
Define the state variation
\[
  \delta u_i
  := \left.\frac{\mathrm{d}}{\mathrm{d}\varepsilon}
        u_i(q_0 + \varepsilon v)\right|_{\varepsilon=0}.
\]
Then $\delta u_i$ satisfies \eqref{eq:linstate}. 
Since the sensor configuration $S$ is fixed, the measurement operator
$M_i(S)$ is independent of $q$, and the resulting variation of the
measurements for load case $i$ is
\[
  (Jv)_i
  := \left.\frac{\mathrm{d}}{\mathrm{d}\varepsilon}
             y_i(S,q_0 + \varepsilon v)\right|_{\varepsilon=0}
   = M_i(S)\,\delta u_i.
\]
Stacking over all load cases yields $Jv\in\mathbb{R}^{N_y}$.
Thus, to compute $Jv$ we:
\begin{enumerate}[leftmargin=2em]
  \item assemble the right-hand sides $C_i v$ for $i=1,\dots,n$,
  \item solve $K(\alpha,\beta)\,\delta u_i = C_i v$,
  \item form $(Jv)_i = M_i(S)\,\delta u_i$ and stack.
\end{enumerate}
This requires $n$ linear solves with the stiffness matrix $K(\alpha,\beta)$.

\medskip\noindent
\textbf{2. Adjoint solve for $J^\top w$.}
Given $w\in\mathbb{R}^{N_y}$, decompose it as
$w=(w_1,\dots,w_n)$, where $w_i$ collects the components associated
with load case $i$.

To derive the adjoint representation of $J^\top w$, let
$w = (w_1,\dots,w_n)$ collect the contributions from all load cases and recall
that for any $v\in\mathbb{R}^{N_q}$,
\[
  Jv
  = \left.\frac{\mathrm{d}}{\mathrm{d}\varepsilon}
           y(S,q_0 + \varepsilon v)\right|_{\varepsilon=0}.
\]
For load case $i$, the corresponding state and measurement variations satisfy
\[
  K(\alpha,\beta)\,\delta u_i = C_i v,
  \qquad
  (Jv)_i = M_i(S)\,\delta u_i,
\]
with $C_i$ defined in \eqref{eq:linstate}.
Thus
\[
  w^\top (Jv)
  = \sum_{i=1}^n w_i^\top (Jv)_i
  = \sum_{i=1}^n (M_i(S)^\top w_i)^\top \delta u_i.
\]
Introduce adjoint variables $\lambda_i \in \mathbb{R}^{N_u}$ solving
\begin{equation}
  K(\alpha,\beta)^\top \lambda_i = M_i(S)^\top w_i.
  \label{eq:adjoint-op-deriv}
\end{equation}
Then
\[
  (M_i(S)^\top w_i)^\top \delta u_i
  = (K(\alpha,\beta)^\top \lambda_i)^\top \delta u_i
  = \lambda_i^\top K(\alpha,\beta)\,\delta u_i
  = \lambda_i^\top C_i v.
\]
Summing over load cases yields
\[
  w^\top (Jv)
  = \sum_{i=1}^n \lambda_i^\top C_i v
  = \sum_{i=1}^n  \Big( C_i^\top \lambda_i \Big)^\top v
  = \Big( \sum_{i=1}^n C_i^\top \lambda_i \Big)^\top v ,
\]
where we have used the linearity to bring summation inside. 

Since this holds for all $v\in\mathbb{R}^{N_q}$, we conclude that
\begin{equation}
  J^\top w = \sum_{i=1}^n C_i^\top \lambda_i.
  \label{eq:JTw-final}
\end{equation}
This shows that the application of $J^\top$ to a vector $w$ can be evaluated
using $n$ adjoint solves \eqref{eq:adjoint-op-deriv} with the same stiffness matrix
$K(\alpha,\beta)$.

\medskip\noindent
\textbf{3. Applying $F = J^\top R^{-1} J$.}
For any $z\in\mathbb{R}^{N_q}$ we have
\[
  Fz = J^\top (R^{-1} Jz).
\]
Using the procedures above to evaluate $Jz$ and $J^\top(\cdot)$, we obtain
$Fz$ without forming $J$ or $F$ explicitly:
\[
  y = Jz,\qquad
  \tilde y = R^{-1} y,\qquad
  Fz = J^\top \tilde y.
\]
Hence, one application of $F$ requires $n$ solves with $K(\alpha,\beta)$ and
$n$ solves with $K(\alpha,\beta)^\top$.

\medskip\noindent
\textbf{Notes.}
\begin{itemize}[leftmargin=2em]

\item[(i)]  
The matrices $C_i$ in~\eqref{eq:linstate} contain the parameter-to-RHS 
sensitivities
\[
  C_i
  = 
    \frac{\partial f_{\Delta T}}{\partial q}
    -\frac{\partial K(\alpha,\beta)}{\partial q}\,u_i,
\]
and are assembled elementwise using the same derivatives of $K$, $f_i$, and 
$f_{\Delta T}$ that appear in the inverse solver.  
No additional model derivatives are required.

\item[(ii)]  
When $N_q$ is modest (small number of unknown parameters), it is practical to 
assemble the Fisher matrix
\[
  F = J^\top R^{-1} J
\]
explicitly and apply a dense factorization (e.g.\ Cholesky) to evaluate 
$\log\det F$ and $F^{-1}$.

\item[(iii)]
When $N_q$ is large (e.g.\ spatially distributed parameters), we keep
$F = J^\top R^{-1} J$ matrix-free and access it only through the operator
application
\[
  Fz = J^\top R^{-1}(Jz),
\]
which uses one forward and one adjoint sweep per product.  
This is sufficient to:
\begin{itemize}
  \item solve linear systems $Fx = g$ by a Krylov method (e.g.\ CG) whenever
        $F^{-1}g$ is needed (A-optimality or explicit
        formation of $B = R^{-1}JF^{-1}$), and
  \item approximate the D-optimal objective via stochastic Lanczos quadrature,
        using the identity $\log\det F = \operatorname{tr}(\log F)$ and
        repeated matrix--vector products with~$F$.
\end{itemize}

\item[(iv)]  
The expressions for $Jv$ and $J^\top w$ depend only on derivatives with 
respect to the \emph{parameters} $q$ and do not involve derivatives with 
respect to the \emph{sensor design variables} $S$.  
Sensitivity of the D-optimal objective
\[
  \Phi_{\mathrm{D}}(S) = -\log\det F(S)
\]
with respect to $S$ requires a second layer of calculus to evaluate 
$\mathrm{d}J(S)$; this is derived in the next subsection.

\end{itemize}

\subsection{Sensitivities of the D-optimal criterion with respect to sensor design}

We now derive expressions for the derivative of the D-optimal objective
\[
  \Phi_{\mathrm{D}}(S)
  = -\log\det F(S),
  \qquad
  F(S) = J(S)^\top R^{-1} J(S),
\]
with respect to continuous design parameters that define the sensor
configuration $S$.
Typical example include sensor positions (for pointwise or interpolated measurements).

Recall from \eqref{eq:fisher-final} and \eqref{eq:dPhi_BdJ} that the first
variation of $\Phi_{\mathrm{D}}$ with respect to a perturbation $\mathrm{d}J(S)$
is
\begin{equation}
  \mathrm{d}\Phi_{\mathrm{D}}
  = -2\,\operatorname{tr}\!\big( B(S)^\top\, \mathrm{d}J(S) \big),
  \qquad
  B(S) := R^{-1} J(S) F(S)^{-1} \in \mathbb{R}^{N_y \times N_q}.
  \label{eq:dPhi_BdJ-again}
\end{equation}
Here $B(S)$ is the ``adjoint weight'' matrix and depends on $S$ only through
$J(S)$ and $F(S)$.
Once Jacobian $J(S)$ and Fisher-information operator $F(S)$ actions are available, either through explicit assembly (for moderate parameter dimension)
or through matrix-free linear solves with $F$, $B(S)$ can be evaluated by solving linear systems with $F(S)$ or by explicit multiplication if $F(S)$ is
assembled.

\subsubsection{Block structure and dependence on $S$}

We first make explicit how $J(S)$ depends on the sensor configuration through
the measurement operators $M_i(S)$.
From the linearized state equation \eqref{eq:linstate}, we have
\[
  K(\alpha,\beta)\,U_i = C_i,
  \qquad
  C_i
  = 
    \frac{\partial f_{\Delta T}}{\partial q}
    -\frac{\partial K(\alpha,\beta)}{\partial q}\,u_i,
\]
so that the Jacobian block for load case $i$ is
\[
  J_i(S) = M_i(S)\,U_i \in \mathbb{R}^{m_i \times N_q},
\]
where $m_i$ denotes the number of measurements for load case $i$.
Stacking over all load cases,
\[
  J(S)
  =
  \begin{bmatrix}
    J_1(S) \\
    \vdots \\
    J_n(S)
  \end{bmatrix}
  =
  \begin{bmatrix}
    M_1(S)\,U_1 \\
    \vdots \\
    M_n(S)\,U_n
  \end{bmatrix}
  \in \mathbb{R}^{N_y \times N_q}.
\]
Note that $U_i$ depends on the parameter vector $q$ but is independent of the
sensor configuration $S$ under our ``passive sensor'' assumption: $S$ affects
only the measurement operators $M_i(S)$ and not the PDE itself.

Similarly, we partition the adjoint weight matrix $B(S)$ into blocks
corresponding to each load case:
\[
  B(S)
  =
  \begin{bmatrix}
    \mathcal{B}_1(S) \\
    \vdots \\
    \mathcal{B}_n(S)
  \end{bmatrix},
  \qquad
  \mathcal{B}_i(S) \in \mathbb{R}^{m_i \times N_q}.
\]
Then
\[
  \operatorname{tr}\!\big( B(S)^\top \mathrm{d}J(S) \big)
  = \sum_{i=1}^n \operatorname{tr}\!\big(
      \mathcal{B}_i(S)^\top\,\mathrm{d}J_i(S)
    \big),
\]
and~\eqref{eq:dPhi_BdJ-again} becomes
\begin{equation}
  \mathrm{d}\Phi_{\mathrm{D}}
  = -2 \sum_{i=1}^n
      \operatorname{tr}\!\big(
        \mathcal{B}_i(S)^\top\,\mathrm{d}J_i(S)
      \big).
  \label{eq:dPhi-sum}
\end{equation}

Since $J_i(S) = M_i(S) U_i$ and $U_i$ is independent of $S$, the variation of
$J_i(S)$ induced by a perturbation of the design $S$ is
\[
  \mathrm{d}J_i(S)
  = (\mathrm{d}M_i(S))\,U_i.
\]
Substituting this into \eqref{eq:dPhi-sum} gives
\begin{equation}
  \mathrm{d}\Phi_{\mathrm{D}}
  = -2 \sum_{i=1}^n
      \operatorname{tr}\!\big(
        \mathcal{B}_i(S)^\top\,(\mathrm{d}M_i(S))\,U_i
      \big).
  \label{eq:dPhi-dM}
\end{equation}
This identity shows that once $U_i$ and the block weights $\mathcal{B}_i(S)$
are available, the sensitivity of the D-optimal criterion with respect to
sensor design is governed entirely by the variations of the measurement
operators $M_i(S)$.

\subsubsection{Derivative with respect to a scalar design parameter}

Let $\theta$ be a scalar design parameter (e.g.\ a weight for a particular
candidate sensor, or one component of a sensor position vector). 
We assume that the dependence $S \mapsto M_i(S)$ is sufficiently smooth so that
the derivatives $\partial M_i(S)/\partial \theta$ exist for the design
parameters of interest.
Then $\mathrm{d}M_i(S)$ in \eqref{eq:dPhi-dM} becomes
\[
  \mathrm{d}M_i(S) = \frac{\partial M_i(S)}{\partial \theta}\,\mathrm{d}\theta,
\]
and we obtain the directional derivative
\[
  \mathrm{d}\Phi_{\mathrm{D}}
  = \frac{\partial \Phi_{\mathrm{D}}}{\partial \theta}\,\mathrm{d}\theta
  = -2 \sum_{i=1}^n
      \operatorname{tr}\!\big(
        \mathcal{B}_i(S)^\top\,
        \frac{\partial M_i(S)}{\partial \theta}\,
        U_i
      \big)\,\mathrm{d}\theta.
\]
Hence the derivative of the D-optimal objective with respect to $\theta$ is
\begin{equation}
  \frac{\partial \Phi_{\mathrm{D}}}{\partial \theta}
  = -2 \sum_{i=1}^n
      \operatorname{tr}\!\big(
        \mathcal{B}_i(S)^\top\,
        \frac{\partial M_i(S)}{\partial \theta}\,
        U_i
      \big).
  \label{eq:grad-theta}
\end{equation}

Formula \eqref{eq:grad-theta} is completely general and covers a wide range of
sensor design parametrizations:
\begin{itemize}[leftmargin=2em]
  \item \emph{Sensor weights.}  
    If $M_i(S)$ is constructed as a weighted sum of candidate sensor rows,
    $M_i(S) = \sum_j \theta_j M_{i,j}$, then
    $\partial M_i(S)/\partial \theta_j = M_{i,j}$ and
    \eqref{eq:grad-theta} reduces to a simple weighted Frobenius inner
    product of $\mathcal{B}_i(S)$ and $M_{i,j} U_i$.
  \item \emph{Sensor positions.}  
    For point sensors defined by interpolation from nodal values,
    $M_i(S)$ consists of rows of shape functions evaluated at sensor
    coordinates $x_{\mathrm{sens}}(\theta)$.
    Then $\partial M_i(S)/\partial \theta$ involves gradients of the shape
    functions and the Jacobian $\partial x_{\mathrm{sens}}/\partial \theta$,
    which are standard ingredients in finite element codes.
\end{itemize}
In all cases, no additional PDE solves are required beyond those used to
compute $U_i$, $J(S)$, $F(S)$, and $B(S)$.
The design gradient \eqref{eq:grad-theta} is obtained by inexpensive local
operations on the measurement operators and the already-computed sensitivity
matrices $U_i$.

\begin{remark}[On the role of adjoints and parameter dimension]
\rm 
In inverse problems one typically seeks the gradient of a scalar misfit
functional with respect to the parameter vector $q$, and adjoint techniques
are used to avoid forming the Jacobian $J(S)=\partial y/\partial q$.
In contrast, D-optimal sensor design depends on the full Fisher information
matrix $F(S) = J(S)^\top R^{-1} J(S)$, which aggregates sensitivities in
\emph{all} parameter directions.
Thus, exact evaluation of $F(S)$ (and its log-determinant) inherently requires
access to the sensitivities in all components of $q$, for example through the
state sensitivity matrices $U_i$ solving $K U_i = C_i$.
In this work we focus on parameter vectors of moderate dimension, for which
storing $U_i$ and $F(S)$ is feasible.
For very high-dimensional parameters, additional approximations (e.g. reduced
parameter subspaces or randomized low-rank approximations of $F$) would be
required and are left for future work.
\end{remark}

\subsubsection{Summary of steps for gradient evaluation}

For a given sensor configuration $S$ and parameter vector $q_0$:

\begin{enumerate}[leftmargin=2em]
  \item For each load case $i=1,\dots,n$:
    \begin{enumerate}
      \item Solve the forward problem
        $K(\alpha,\beta)\,u_i(q_0) = f_i(q_0) + f_{\Delta T}(q_0)$.
      \item Assemble $C_i$ as in \eqref{eq:linstate} and solve the linearized
            state equation $K(\alpha,\beta)\,U_i = C_i$ to obtain
            $U_i = \partial u_i/\partial q$.
      \item Form the Jacobian block $J_i(S) = M_i(S)\,U_i$.
    \end{enumerate}
  \item Assemble the Fisher matrix $F(S) = J(S)^\top R^{-1} J(S)$ and compute
        the D-optimal objective $\Phi_{\mathrm{D}}(S) = -\log\det F(S)$.
  \item Compute the adjoint weight matrix $B(S) = R^{-1} J(S) F(S)^{-1}$ and
        extract its load-case blocks $\mathcal{B}_i(S)$.
  \item For each scalar design parameter $\theta$ of interest, evaluate the
        derivatives $\partial M_i(S)/\partial \theta$ and form
        $\partial \Phi_{\mathrm{D}}/\partial \theta$ using \eqref{eq:grad-theta}.
\end{enumerate}

This procedure yields the gradient of the D-optimal sensor design objective
with respect to arbitrary smooth parametrizations of the sensor configuration,
using only linear algebra and model ingredients already present in the inverse
problem solver.

\section{Didactic example: one load, one sensor, elementwise weaknesses}

We now illustrate how the Fisher-information framework recovers a simple
``closed-form'' sensor placement rule in a simplified setting.
Consider a single load case with parameters
\[
  \alpha = (\alpha_1,\dots,\alpha_{N_e}) \in [\alpha_{\rm min},1]^{N_e},
\]
where $\alpha_e$ scales the stiffness of element $e$.  Let
$u(\alpha)\in\mathbb{R}^{N_u}$ denote the corresponding finite element
state, solving
\[
  K(\alpha)\,u(\alpha) = f,
\]
for a fixed right-hand side $f$.  We assume a reference parameter
$\alpha_0$ (e.g.\ the undamaged configuration) and consider small
perturbations around $\alpha_0$.

Consider a set of candidate point (or local) sensors indexed by
$j=1,\dots,m$.  For a given sensor $j$, the scalar measurement is
\[
  y_j(\alpha) = M_j\,u(\alpha),
\]
where $M_j\in\mathbb{R}^{1\times N_u}$ selects or interpolates a
displacement/strain component at the sensor location $x_j$.
Let
\[
  U := \frac{\partial u(\alpha)}{\partial \alpha}
       \Big|_{\alpha=\alpha_0}
  \in \mathbb{R}^{N_u\times N_e}
\]
denote the state sensitivity matrix, whose $e$-th column is the
sensitivity of the state to the stiffness scaling of element $e$.
Then the Jacobian row of sensor $j$ with respect to the elementwise
weaknesses is
\[
  J_j
  := \frac{\partial y_j(\alpha)}{\partial \alpha}
     \Big|_{\alpha=\alpha_0}
  = M_j U \in \mathbb{R}^{1\times N_e}.
\]
The entry $(J_j)_e$ represents the sensitivity of the $j$-th sensor to
a perturbation of element $e$.

We assume additive scalar Gaussian noise on sensor $j$,
\[
  y_j^{\mathrm{obs}}
  = y_j(\alpha^\star) + \eta,
  \qquad
  \eta \sim \mathcal{N}(0,\sigma^2),
\]
for some unknown true parameter vector $\alpha^\star$ and known
variance~$\sigma^2>0$.

\begin{proposition}[One load, one sensor, elementwise weaknesses]
\label{prop:one-sensor}
Suppose that at most one element is weakened, i.e.\ we are interested in
detecting deviations of a single scalar parameter $\alpha_k$ from its
reference value while all other components of $\alpha$ are held fixed.
For sensor $j$, the Fisher information for the scalar parameter
$\alpha_k$ is
\[
  F_{j,k} = \frac{1}{\sigma^2}
            \bigg(\frac{\partial y_j}{\partial \alpha_k}\Big|_{\alpha_0}\bigg)^2
          = \frac{1}{\sigma^2}\,(J_j)_k^2.
\]
If we seek a single sensor location that is robust across all possible
single-element weakenings, a natural design criterion is to maximize the
average Fisher information
\[
  \sum_{k=1}^{N_e} F_{j,k}.
\]
This yields the optimal one-sensor location
\begin{equation}
  j^\star \in \arg\max_{j=1,\dots,m} \|J_j\|_2^2
           = \arg\max_{j=1,\dots,m}
             \sum_{k=1}^{N_e}
             \bigg(\frac{\partial y_j}{\partial \alpha_k}\Big|_{\alpha_0}\bigg)^2.
  \label{eq:one-sensor-rule}
\end{equation}
In words, the optimal sensor is placed where the measurement is
simultaneously most sensitive (in an $\ell_2$ sense) to the collection
of elementwise weaknesses.
\end{proposition}

\begin{proof}
Fix a sensor index $j$ and an element index $k$.  Consider the scalar
parameter $\theta := \alpha_k$ while all other components of $\alpha$
are held fixed at their reference values.  For small perturbations
around $\alpha_0$, the measurement map is well approximated by its
first-order Taylor expansion
\[
  y_j(\theta)
  \approx y_j(\alpha_0)
          + \frac{\partial y_j}{\partial \alpha_k}\Big|_{\alpha_0}
            (\theta - \alpha_{0,k}).
\]
Define the scalar sensitivity
\[
  s_{j,k}
  := \frac{\partial y_j}{\partial \alpha_k}\Big|_{\alpha_0}
   = (J_j)_k.
\]

We model the measurement at sensor $j$ as the true (noise-free) model
prediction plus additive Gaussian noise,
\[
  y_j^{\mathrm{obs}} = y_j(\theta) + \eta_j,
  \qquad \eta_j \sim \mathcal{N}(0,\sigma^2),
\]
where $y_j(\theta)$ is the deterministic response predicted by the finite
element model for parameter value $\theta$.  
Since $y_j(\theta)$ is non-random for fixed $\theta$, the conditional density of
$y_j^{\mathrm{obs}}$ given $\theta$ is
\[
  p(y_j^{\mathrm{obs}}\mid\theta)
  = \frac{1}{\sqrt{2\pi\sigma^2}}
    \exp\!\left(
      -\frac{(y_j^{\mathrm{obs}} - y_j(\theta))^2}{2\sigma^2}
    \right),
\]
i.e.,
\[
  y_j^{\mathrm{obs}} \mid \theta
  \sim \mathcal{N}\!\big(y_j(\theta),\,\sigma^2\big).
\]

The log-likelihood (up to an additive constant independent of $\theta$) is
\[
  \ell(\theta\,;\,y_j^{\mathrm{obs}})
  = \log p(y_j^{\mathrm{obs}}\mid\theta)
  = -\frac{1}{2\sigma^2}
    \big(y_j^{\mathrm{obs}} - \mu(\theta)\big)^2,
\qquad
  \mu(\theta) := y_j(\theta).
\]
Differentiating with respect to $\theta$ gives
\[
  \frac{\partial \ell}{\partial \theta}
  = \frac{1}{\sigma^2}
    \big(y_j^{\mathrm{obs}} - \mu(\theta)\big)
    \frac{\partial \mu(\theta)}{\partial \theta}.
\]
At the reference point $\theta_0 = \alpha_{0,k}$, the derivative of the
mean is
\[
  \frac{\partial \mu(\theta)}{\partial \theta}\Big|_{\theta_0}
  = \frac{\partial y_j}{\partial \alpha_k}\Big|_{\alpha_0}
  = s_{j,k}.
\]
The Fisher information for the scalar parameter $\theta$ at $\theta_0$
is
\[
  F_{j,k}
  = \mathbb{E}\bigg[
      \bigg(\frac{\partial \ell}{\partial \theta}\Big|_{\theta_0}\bigg)^2
    \bigg]
  = \mathbb{E}\Big[
      \frac{1}{\sigma^4}
      \big(y_j^{\mathrm{obs}} - \mu(\theta_0)\big)^2
      s_{j,k}^2
    \Big],
\]
where the expectation is with respect to $y_j^{\mathrm{obs}}$ distributed
as $\mathcal{N}(\mu(\theta_0),\sigma^2)$.  Since
$\mathbb{E}[(y_j^{\mathrm{obs}} - \mu(\theta_0))^2] = \sigma^2$, we
obtain
\[
  F_{j,k}
  = \frac{1}{\sigma^4}\,\sigma^2\,s_{j,k}^2
  = \frac{1}{\sigma^2}\,s_{j,k}^2
  = \frac{1}{\sigma^2}
    \bigg(\frac{\partial y_j}{\partial \alpha_k}\Big|_{\alpha_0}\bigg)^2
  = \frac{1}{\sigma^2}\,(J_j)_k^2.
\]

Now suppose that the weakened element is unknown and may be any one of
$e=1,\dots,N_e$.  A natural robust design criterion for sensor $j$ is to
maximize the \emph{average} Fisher information over these possibilities,
\[
  \sum_{k=1}^{N_e} F_{j,k}
  = \frac{1}{\sigma^2} \sum_{k=1}^{N_e} (J_j)_k^2
  = \frac{1}{\sigma^2} \|J_j\|_2^2.
\]
Since $\sigma^2$ is common to all sensors, it does not affect the
argmax.  Therefore, the optimal single sensor location in this average
sense is
\[
  j^\star \in \arg\max_{j=1,\dots,m} \|J_j\|_2^2,
\]
which is precisely~\eqref{eq:one-sensor-rule}.
\end{proof}

\begin{remark}
\rm 
Proposition~\ref{prop:one-sensor} shows that, in the simple setting of
a single load case, one sensor, and elementwise stiffness scalings, the
general Fisher-information framework reduces to a transparent rule:
\emph{place the sensor where the measurement is simultaneously most
sensitive to the collection of elementwise weaknesses}.
In practice, this can be implemented by computing the row norms of the
Jacobian $J(S)$ (or equivalently of the state sensitivity matrix $U$),
and placing the sensor at the candidate location with maximal row norm.
This didactic example provides a useful sanity check for the more
general multi-load, multi-sensor formulations. 
\end{remark}

\begin{proposition}[Detectability-aware average Fisher criterion]
\label{prop:thresholded-fisher}
In the setting of Proposition~\ref{prop:one-sensor}, suppose that:
\begin{itemize}[leftmargin=2em]
  \item a minimal element-wise stiffness loss $\Delta\alpha_{\min}>0$ is of practical
        interest (e.g.\ a $5\%$ reduction), and
  \item the sensor cannot reliably distinguish measurement changes smaller than
        a resolution (or noise floor) $\delta_y>0$.
\end{itemize}
Then a perturbation $\Delta\alpha_k$ of element $k$ is practically detectable
by sensor $j$ only if
\[
  \big|\Delta y_{j,k}\big|
  \;\approx\;
  \bigg|\frac{\partial y_j}{\partial \alpha_k}\Big|_{\alpha_0}\bigg|
  \Delta\alpha_{\min}
  \;\gtrsim\; \delta_y.
\]
Equivalently, in terms of the scalar Fisher information
\[
  F_{j,k}
  = \frac{1}{\sigma^2}
    \bigg(\frac{\partial y_j}{\partial \alpha_k}\Big|_{\alpha_0}\bigg)^2,
\]
detectability requires
\begin{equation}
  F_{j,k} \;\gtrsim\; F_{\min}
  := \frac{1}{\sigma^2}
     \bigg(\frac{\delta_y}{\Delta\alpha_{\min}}\bigg)^2.
  \label{eq:Fmin-def}
\end{equation}
A detectability-aware average Fisher criterion for sensor $j$ is then
\begin{equation}
  \mathcal{J}_j^{\mathrm{trunc}}
  := \sum_{k=1}^{N_e} \max\{F_{j,k} - F_{\min},\,0\},
  \label{eq:J-trunc-def}
\end{equation}
and an optimal one-sensor location in this sense is
\begin{equation}
  j^\star \in \arg\max_{j=1,\dots,m} \mathcal{J}_j^{\mathrm{trunc}}.
  \label{eq:J-trunc-opt}
\end{equation}
\end{proposition}

\begin{proof}
By the first-order approximation,
\[
  \Delta y_{j,k}
  \;\approx\;
  \frac{\partial y_j}{\partial \alpha_k}\Big|_{\alpha_0}\,\Delta\alpha_{\min}.
\]
Requiring $\big|\Delta y_{j,k}\big|\gtrsim\delta_y$ gives
\[
  \bigg|
    \frac{\partial y_j}{\partial \alpha_k}\Big|_{\alpha_0}
  \bigg|
  \;\gtrsim\;
  \frac{\delta_y}{\Delta\alpha_{\min}}.
\]
Squaring both sides and multiplying by $1/\sigma^2$ yields
\[
  \frac{1}{\sigma^2}
  \bigg(
    \frac{\partial y_j}{\partial \alpha_k}\Big|_{\alpha_0}
  \bigg)^2
  \;\gtrsim\;
  \frac{1}{\sigma^2}
  \bigg(\frac{\delta_y}{\Delta\alpha_{\min}}\bigg)^2
  = F_{\min},
\]
which is precisely \eqref{eq:Fmin-def}.  Hence elements with
$F_{j,k} < F_{\min}$ correspond to parameter perturbations that produce
measurement changes below the effective resolution and are practically
undetectable by sensor $j$.

To discount such contributions in the design criterion, we replace the
raw sum $\sum_k F_{j,k}$ by the truncated sum
\[
  \mathcal{J}_j^{\mathrm{trunc}}
  = \sum_{k=1}^{N_e} \max\{F_{j,k} - F_{\min},\,0\}.
\]
For any $k$ with $F_{j,k} \le F_{\min}$, the contribution is zero; for
$F_{j,k} > F_{\min}$, only the excess over the detectability threshold
$F_{\min}$ is counted.  Maximizing $\mathcal{J}_j^{\mathrm{trunc}}$ over
$j=1,\dots,m$ therefore selects sensors that are not only sensitive in
an average Fisher sense, but also provide practically detectable signal
levels for as many elements as possible.  This yields the optimality
condition \eqref{eq:J-trunc-opt}.
\end{proof}

\begin{remark}[Count-based and smooth variants]
\rm
The truncated average Fisher criterion in \eqref{eq:J-trunc-def} still
weights elements by how far above the detectability threshold $F_{\min}$
they are.  In some situations it is more natural to \emph{ignore} this
magnitude and simply count how many elements are detectable at all.

To this end, we introduce the count-based criterion
\[
  \mathcal{J}_j^{\mathrm{count}}
  := \sum_{k=1}^{N_e}
      \mathbf{1}\{F_{j,k} \ge F_{\min}\},
\]
where $\mathbf{1}\{\cdot\}$ is the indicator function.  For a fixed
sensor location $j$, the quantity $\mathcal{J}_j^{\mathrm{count}}$ is
precisely the number of elements $k$ for which the Fisher information
$F_{j,k}$ exceeds the minimal detectable level $F_{\min}$.  Thus
$\mathcal{J}_j^{\mathrm{count}}$ selects the sensor that can
\emph{reliably see} the largest number of potentially weakened elements,
which matches the heuristic ``place the sensor where it can see the
largest number of elements.''
\end{remark}

\paragraph{Bridge to examples.}
The scalar Fisher information formula and the one-sensor rule provide a simple and powerful way to interpret 
sensor placement principles.  When each parameter represents a localized effect 
(e.g.\ elementwise stiffness~$\alpha_k$), the row 
$J_j$ describes how the measurement at sensor $j$ responds to perturbations of 
each location-specific parameter.  Thus, $\|J_j\|_2^2$ captures how many 
parameters sensor~$j$ is sensitive to and how strongly.

In the examples below, we compute $J$ explicitly for several classical models 
(1D bar, strain sensors, etc.).  
Because these systems are simple, the Fisher information and the optimal sensor 
locations can be obtained in \emph{closed form}, making the theoretical rules 
directly interpretable for engineering intuition.

\subsection{Worked example: 1D bar with 10 elements}

We illustrate the ideas of Proposition~\ref{prop:one-sensor} in a simple
one-dimensional setting where all quantities can be written in closed
form.

Consider an axially loaded bar of length $L>0$ with constant cross-sectional
area $A$ and Young's modulus $E$.  We discretize the bar into
$N_e = 10$ equal elements of length
\[
  \ell = L/N_e.
\]
Let $x_0=0<x_1<\cdots<x_{10}=L$ denote the node positions, with
$x_j = j\ell$.
We assume the left end is clamped and the right end is subjected to a
tensile load $P>0$:
\[
  u(0) = 0, \qquad EA\,u'(L) = P.
\]
Let the candidate displacement sensors are located at the nodes
$x_j$, $j=1,\dots,10$, i.e.\ at all nodes except the clamped left end
$x_0$ where $u_0 = 0$ is prescribed and therefore uninformative for
damage detection.

\paragraph{Elementwise weakening parameters.}
We introduce elementwise stiffness scaling coefficients
\[
  \alpha = (\alpha_1,\dots,\alpha_{10}) \in [\alpha_{\rm min},1]^{10},
\]
with $\alpha_e=1$ corresponding to the undamaged element $e$ and
$\alpha_e<1$ representing weakening (loss of stiffness) on that element.
In a standard two-node linear finite element discretization, each element
$e$ acts as a spring of stiffness
\[
  k_e = \alpha_e\,\frac{EA}{\ell},
\]
and the bar behaves as a series connection of these springs under the
end load $P$.

The extension of element $e$ under the load $P$ is
\[
  \delta_e
  = \frac{P}{k_e}
  = \frac{P\ell}{\alpha_e EA},
\]
so that the displacement at node $j$ (counted from the clamped left end)
is the cumulative extension of the first $j$ elements:
\begin{equation}
  u_j(\alpha)
  := u(x_j)
  = \sum_{e=1}^j \delta_e
  = \frac{P\ell}{EA} \sum_{e=1}^j \frac{1}{\alpha_e},
  \qquad j = 0,1,\dots,10,
  \label{eq:1D-u-alpha}
\end{equation}
with $u_0 = 0$.

\paragraph{Sensor model and Jacobian.}
We consider candidate displacement sensors located at the nodes
$x_j$, $j=1,\dots,10$.  A sensor placed at node $x_j$ measures
\[
  y_j(\alpha) := u_j(\alpha),
\]
so the measurement operator for sensor $j$ is simply the row vector
$M_j = e_j^\top$, where $e_j$ is the $j$-th standard basis vector in
$\mathbb{R}^{11}$ (the full nodal displacement vector).

We take the reference parameter to be the undamaged bar,
\[
  \alpha_0 = (1,\dots,1),
\]
and linearize the measurements around $\alpha_0$.
Differentiating \eqref{eq:1D-u-alpha} with respect to $\alpha_k$ gives
\[
  \frac{\partial u_j(\alpha)}{\partial \alpha_k}
  = \frac{P\ell}{EA} \frac{\partial}{\partial \alpha_k}
    \Big( \sum_{e=1}^j \frac{1}{\alpha_e} \Big)
  =
  \begin{cases}
    -\,\dfrac{P\ell}{EA}\,\dfrac{1}{\alpha_k^2}, & k \le j,\\[6pt]
    0, & k > j.
  \end{cases}
\]
Evaluating at $\alpha=\alpha_0$ (i.e.\ $\alpha_k=1$) yields
\begin{equation}
  \frac{\partial u_j}{\partial \alpha_k}\Big|_{\alpha_0}
  =
  \begin{cases}
    -\,c, & k \le j,\\[4pt]
    0, & k > j,
  \end{cases}
  \qquad
  c := \dfrac{P\ell}{EA}.
  \label{eq:1D-sens-entry}
\end{equation}
Thus, the Jacobian row for sensor $j$ is
\begin{equation}
  J_j
  := \frac{\partial y_j}{\partial \alpha}\Big|_{\alpha_0}
  = \big( (J_j)_1,\dots,(J_j)_{10} \big)
  = \big( -c,\dots,-c,0,\dots,0 \big),
  \label{eq:1D-J-row}
\end{equation}
with the first $j$ entries equal to $-c$ and the remaining $10-j$
entries equal to zero.
In other words, a sensor at node $x_j$ is sensitive to weakening in
elements $1,\dots,j$, but completely insensitive to elements
$k>j$ that lie to its right.

\paragraph{Average Fisher information for single-element weakening.}
As in Proposition~\ref{prop:one-sensor}, we now assume that at most one
element is weakened at a time, i.e.\ we consider scalar perturbations
of individual $\alpha_k$ while all others remain fixed.  We model the
sensor noise at node $x_j$ as
\[
  y_j^{\mathrm{obs}}
  = y_j(\alpha^\star) + \eta,
  \qquad
  \eta \sim \mathcal{N}(0,\sigma^2).
\]
For fixed element index $k$, the scalar Fisher information for $\alpha_k$
obtained from sensor $j$ is
\[
  F_{j,k}
  = \frac{1}{\sigma^2}
    \bigg(\frac{\partial y_j}{\partial \alpha_k}\Big|_{\alpha_0}\bigg)^2
  = \frac{1}{\sigma^2} (J_j)_k^2.
\]
Using \eqref{eq:1D-sens-entry}, this becomes
\[
  F_{j,k}
  =
  \begin{cases}
    \dfrac{c^2}{\sigma^2}, & k \le j,\\[6pt]
    0, & k > j.
  \end{cases}
\]
If we aim to design a single sensor that is, on average, informative
for detecting weakening in any one of the ten elements, a natural
criterion is the average Fisher information
\[
  \sum_{k=1}^{10} F_{j,k}
  = \frac{c^2}{\sigma^2} \sum_{k=1}^{10}
    \mathbf{1}\{k \le j\}
  = \frac{c^2}{\sigma^2}\, j.
\]
Since $c^2/\sigma^2$ is independent of $j$, the optimal sensor location
in this sense is
\[
  j^\star \in \arg\max_{j=1,\dots,10} j
  = \{10\},
\]
i.e.\ the rightmost node $x_{10} = L$.
At this location, the sensor is sensitive to weakening in \emph{all}
ten elements, whereas a sensor at node $x_j$ only ``sees'' elements
$1,\dots,j$.

This explicit computation confirms the intuitive rule discussed in
Proposition~\ref{prop:one-sensor}: in this simple 1D bar, placing the
sensor at the location with the largest displacement---here, the free
end---maximizes the number of elements that are detectable via that
sensor and maximizes the average Fisher information for single-element
weakening.

\begin{remark}[Thresholded variants]
\rm 
If the sensor has a finite resolution $\delta_y$ and only stiffness
losses $\Delta\alpha_{\min}$ above a certain size are of interest,
one can introduce a detectability threshold $F_{\min}$ as in
Proposition~\ref{prop:thresholded-fisher} and count only those elements
for which $F_{j,k} \ge F_{\min}$.  In the present example, for a given
sensor position $j$, either all elements $1,\dots,j$ are detectable
(if $c^2/\sigma^2 \ge F_{\min}$), or none are.  Thus the thresholded
criteria reduce to the same rule: place the sensor at the free end
$x_{10}$ to maximize the number of elements whose weakening is both
seen by the sensor and above the detectability threshold.
\end{remark}

\subsection{1D bar with one strain sensor}

We now revisit the 1D bar of the previous subsection, but instead of
measuring nodal displacements we place a single \emph{strain} sensor on
one element.

\paragraph{Element strains.}
The extension of element $e$ under the load $P$ is
\[
  \delta_e(\alpha)
  = \frac{P}{k_e}
  = \frac{P\ell}{\alpha_e EA},
\]
so that the (axial) strain in element $e$ is
\begin{equation}
  \varepsilon_e(\alpha)
  = \frac{\delta_e(\alpha)}{\ell}
  = \frac{P}{\alpha_e EA}.
  \label{eq:1D-strain-alpha}
\end{equation}
Thus $\varepsilon_e$ depends only on its own stiffness factor $\alpha_e$:
weakening in other elements does not enter the strain in element $e$ in
this simple series model.

\paragraph{Single strain sensor.}
We now consider a single strain sensor placed in element $r$,
$1 \le r \le N_e$, which measures
\[
  y_r(\alpha) := \varepsilon_r(\alpha)
  = \frac{P}{\alpha_r EA}.
\]
As before, we take the reference parameter to be the undamaged bar,
\[
  \alpha_0 = (1,\dots,1),
\]
and linearize $y_r(\alpha)$ with respect to $\alpha$ at $\alpha_0$.

From \eqref{eq:1D-strain-alpha}, the partial derivatives are
\[
  \frac{\partial y_r(\alpha)}{\partial \alpha_k}
  = \frac{\partial}{\partial \alpha_k}
    \Big( \frac{P}{\alpha_r EA} \Big)
  =
  \begin{cases}
    -\,\dfrac{P}{EA}\,\dfrac{1}{\alpha_r^2}, & k=r,\\[6pt]
    0, & k\neq r,
  \end{cases}
\]
since $y_r$ depends only on $\alpha_r$.
Evaluating at $\alpha=\alpha_0$ (so $\alpha_r=1$) yields
\begin{equation}
  \frac{\partial y_r}{\partial \alpha_k}\Big|_{\alpha_0}
  =
  \begin{cases}
    -\,c_s, & k=r,\\[4pt]
    0, & k\neq r,
  \end{cases}
  \qquad
  c_s := \dfrac{P}{EA}.
  \label{eq:1D-strain-sens-entry}
\end{equation}
Thus the Jacobian row associated with the single strain sensor in
element $r$ is
\begin{equation}
  J_r
  := \frac{\partial y_r}{\partial \alpha}\Big|_{\alpha_0}
  = (0,\dots,0,\underbrace{-c_s}_{k=r},0,\dots,0)
  \in \mathbb{R}^{1\times N_e},
  \label{eq:1D-strain-J-row}
\end{equation}
i.e., a vector with a single nonzero entry at position $k=r$.

In particular, the sensor in element $r$ is sensitive only to weakening
in that element and completely insensitive to weakening in all other
elements, in contrast to the displacement sensor at node $x_j$, whose
response depends on all elements $1,\dots,j$.

\paragraph{Fisher information for single-element weakening.}
As in Proposition~\ref{prop:one-sensor}, suppose that at most one
element is weakened at a time, and we are interested in detecting
deviations of a single scalar parameter $\alpha_k$ from its reference
value, with all other components held fixed.

For fixed $k$, the scalar Fisher information for $\alpha_k$ provided by
this single measurement is (cf.\ Proposition~\ref{prop:one-sensor})
\[
  F_{r,k}
  = \frac{1}{\sigma^2}
    \bigg(
      \frac{\partial y_r}{\partial \alpha_k}\Big|_{\alpha_0}
    \bigg)^2
  = \frac{1}{\sigma^2}(J_r)_k^2.
\]
Using \eqref{eq:1D-strain-sens-entry}, we obtain
\[
  F_{r,k}
  =
  \begin{cases}
    \dfrac{c_s^2}{\sigma^2}, & k=r,\\[6pt]
    0, & k\neq r.
  \end{cases}
\]
In other words, a strain sensor in element $r$ provides nonzero Fisher
information only about the stiffness factor $\alpha_r$ of that same
element; it carries no information about $\alpha_k$ for $k\neq r$.

\paragraph{Average information over unknown weakened element.}
In the displacement-based 1D example, we considered the situation where
\emph{any one} of the $N_e$ elements might be weakened, with all indices
\emph{a priori} equally likely, and used the average Fisher information
\[
  \sum_{k=1}^{N_e} F_{j,k}
\]
as a simple robustness criterion for comparing different single-sensor
locations $j$.
This led to the rule ``place the displacement sensor at the free end''
because that location is simultaneously sensitive to all elements.

For the single strain sensor in element $r$, the analogous average
criterion is
\[
  \mathcal{J}_r^{\mathrm{avg}}
  := \sum_{k=1}^{N_e} F_{r,k}.
\]
Using the expression for $F_{r,k}$ above,
\[
  \mathcal{J}_r^{\mathrm{avg}}
  = \sum_{k=1}^{N_e} F_{r,k}
  = F_{r,r}
  = \frac{c_s^2}{\sigma^2}
  = \frac{1}{\sigma^2}
    \Big(\frac{P}{EA}\Big)^2,
\]
which is \emph{independent} of the element index $r$.

Thus, under the ``single unknown weakened element'' hypothesis and
equal prior weight on each element, all elementwise strain sensor
locations are equivalent in this simple 1D bar: any single strain sensor
provides the same average Fisher information about which element might
be weakened.

\begin{remark}[Comparison with displacement sensing]
\rm
In the displacement-sensor example, the Jacobian row for a sensor at
node $x_j$ had the structure
\[
  J_j = (-c,\dots,-c,0,\dots,0),
\]
so that the sensor at the free end $x_{10}$ was simultaneously
sensitive to the weakening of \emph{all} ten elements, and the average
Fisher information increased with $j$.
In contrast, the strain sensor Jacobian \eqref{eq:1D-strain-J-row} has
exactly one nonzero entry, so each strain sensor ``sees'' only its own
element.
From the point of view of average information over an unknown weakened
element, no particular element is distinguished.

This contrast illustrates a general principle for structural digital
twins: displacement (or rotation) measurements tend to integrate
information over multiple elements, while strain measurements
are local.
In practice, combining a few global displacement sensors with local
strain sensors can provide both wide coverage and sharp localization of
weaknesses.
\end{remark}

\section{Multiple Displacement Sensors on the 1D Bar: Detectability vs.\ Localizability}
\label{sec:multi-sensor-1dbar}

In the single-sensor case we showed that, for the
1D axial bar with elementwise weakening parameters
$\alpha=(\alpha_1,\dots,\alpha_{10})$ and a \emph{single displacement sensor},
the average Fisher-information criterion selects the free end $x_{10}=L$.
This is a \emph{detectability} result: a sensor at node $j$ is sensitive to
weakening in elements $1,\dots,j$, and the free-end displacement is sensitive
to \emph{all} elements.

When \emph{multiple} sensors are available, the design goal changes.
Once a sensor at $x_{10}$ is installed, \emph{every} single-element weakening
remains detectable; additional sensors do not increase the set of detectable
elements.  Their main purpose is to improve \emph{localizability} (spatial
resolution): to distinguish \emph{where} along the bar the weakening occurred.
This section makes this distinction precise and derives a closed-form solution
for the D-optimal placement of multiple displacement sensors in the 1D bar.

\subsection{Setup: Jacobian rows and Fisher matrix for a sensor set}

We consider candidate displacement sensors at the free nodes
$j\in\{1,\dots,10\}$ (node $j=0$ is clamped and measures $u_0\equiv 0$).
A sensor at node $j$ measures
\[
  y_j(\alpha) := u_j(\alpha).
\]
At the reference (undamaged) configuration $\alpha_0=(1,\dots,1)$ we derived
the sensitivities
\begin{equation}
  \frac{\partial u_j}{\partial \alpha_k}\Big|_{\alpha_0}
  =
  \begin{cases}
    -c, & k \le j,\\
    0,  & k > j,
  \end{cases}
  \qquad c:=\dfrac{P\ell}{EA},
  \label{eq:1dbar-sens}
\end{equation}
so that the Jacobian row for sensor $j$ is the ``prefix'' vector
\begin{equation}
  J_j := \frac{\partial y_j}{\partial \alpha}\Big|_{\alpha_0}
       = (-c)\,( \underbrace{1,\dots,1}_{j \text{ entries}},0,\dots,0)
  \in \mathbb{R}^{10}.
  \label{eq:1dbar-Jrow}
\end{equation}

Let $S=\{j_1<\cdots<j_m\}\subset\{1,\dots,10\}$ be a set of $m$ sensor nodes and
let $J(S)\in\mathbb{R}^{m\times 10}$ be the stacked Jacobian with rows
$J_{j_p}$.  Under i.i.d.\ noise with variance $\sigma^2$ (so $R=\sigma^2 I$),
the Fisher information matrix is
\begin{equation}
  F(S) = J(S)^\top R^{-1} J(S)
       = \frac{1}{\sigma^2} J(S)^\top J(S).
  \label{eq:F-gram}
\end{equation}
Thus, up to the constant factor $\sigma^{-2}$, D-optimal design maximizes the
determinant of the Gram matrix $J(S)^\top J(S)$. 

\paragraph{Key observation (localization requires non-redundant rows).}
If two sensors are placed at the same node, they produce identical rows and do
not increase the rank of $J(S)$.  Even if sensors are distinct, if their rows
are nearly linearly dependent, then $\det(J(S)^\top J(S))$ is small and the
configuration provides poor localization.  Hence multiple-sensor design in this
example is fundamentally about choosing sensor nodes so that the prefix vectors
in \eqref{eq:1dbar-Jrow} ``separate'' the elements.

\subsection{Closed form for the D-optimal objective}

The inner products of prefix vectors are explicit.  For $p,q\in\{1,\dots,m\}$,
\[
  J_{j_p}\cdot J_{j_q}
  = c^2 \min(j_p,j_q).
\]
Define the $m\times m$ matrix
\begin{equation}
  K(S) := \big[\min(j_p,j_q)\big]_{p,q=1}^m.
  \label{eq:Kmin}
\end{equation}
Then
\begin{equation}
  J(S)J(S)^\top = c^2 K(S),
  \qquad
  \det\!\big(J(S)J(S)^\top\big) = c^{2m}\det K(S),
  \label{eq:JJt}
\end{equation}
and (since $\det(J^\top J)=\det(JJ^\top)$ for rectangular $J$ of full row rank)
D-optimality for fixed $m$ is equivalent to maximizing $\det K(S)$.

\begin{lemma}[Determinant of the ``min'' matrix]
\label{lem:min-det}
Let $0=j_0<j_1<\cdots<j_m$ and define $K=[\min(j_p,j_q)]_{p,q=1}^m$.
Then
\begin{equation}
  \det K = \prod_{p=1}^m (j_p-j_{p-1}).
  \label{eq:min-det-formula}
\end{equation}
\end{lemma}

\begin{proof}
Let $\Delta_p := j_p-j_{p-1}>0$. Note that $j_p$ denotes the physical node index of the $p$th sensor,
and $\Delta_p := j_p - j_{p-1}$ is the spacing between consecutive sensors.
 
Define the unit lower-triangular matrix
$A\in\mathbb{R}^{m\times m}$ by
\[
  A_{p r} :=
  \begin{cases}
    1, & r\le p,\\
    0, & r>p.
  \end{cases}
\]
Let $D:=\mathrm{diag}(\Delta_1,\dots,\Delta_m)$.  A direct computation shows
that
\[
  (ADA^\top)_{pq} = \sum_{r=1}^{\min(p,q)} \Delta_r
                 = j_{\min(p,q)}
                 = \min(j_p,j_q),
\]
hence $K = A D A^\top$.  Since $A$ is unit lower-triangular, $\det A=1$, so
\[
  \det K = \det(A)\det(D)\det(A^\top) = \det(D) = \prod_{p=1}^m \Delta_p
  = \prod_{p=1}^m (j_p-j_{p-1}),
\]
which is \eqref{eq:min-det-formula}.
\end{proof}

Lemma~\ref{lem:min-det} shows that, for the 1D bar, D-optimal sensor placement
reduces to a discrete ``max-product of increments'' problem.

\subsection{D-optimal sensor locations: equal increments (uniform spacing)}

We now solve the D-optimal placement problem for a fixed number $m$ of sensors.
Let $S=\{j_1<\cdots<j_m\}\subset\{1,\dots,10\}$ and set $j_0:=0$ and
$\Delta_p:=j_p-j_{p-1}\in\mathbb{N}$.
By Lemma~\ref{lem:min-det},
\[
  \det K(S)
  = \prod_{p=1}^m (j_p-j_{p-1})
  = \prod_{p=1}^m \Delta_p,
  \qquad
  \sum_{p=1}^m \Delta_p = j_m \le 10.
\]
In particular, maximizing $\det K(S)$ over sensor sets $S$ is
equivalent to maximizing the product of positive integer increments
$\prod_{p=1}^m \Delta_p$ subject to $\sum_{p=1}^m \Delta_p\le 10$.

Because the product $\prod_{p=1}^m \Delta_p$ strictly increases
when the total sum $\sum_{p=1}^m\Delta_p$ increases (keeping all other increments
fixed and increasing one of them), any maximizer must saturate the constraint
$j_m=\sum_{p=1}^m \Delta_p = 10$.  Hence the free end is always included in a
D-optimal set.

\begin{theorem}[Closed-form D-optimal placement on the 1D bar]
\label{thm:Dopt-1dbar}
Fix $m\in\{1,\dots,10\}$ and consider displacement sensors at nodes
$\{1,\dots,10\}$ with i.i.d.\ noise variance $\sigma^2$.
Let $S=\{j_1<\cdots<j_m\}$ be the chosen sensor set.
Then D-optimal sensor placement (maximizing $\log\det F(S)$) is achieved by
choosing $j_m=10$ and making the increments
\[
  \Delta_p := j_p-j_{p-1},\quad j_0:=0,
\]
as equal as possible, i.e.\ all increments differ by at most $1$.

Concretely, let
\[
  q := \left\lfloor \frac{10}{m} \right\rfloor,
  \qquad
  r := 10 - m q \in \{0,1,\dots,m-1\}.
\]
Then any choice of positive integers $(\Delta_1,\dots,\Delta_m)$ satisfying
\[
  \Delta_p \in \{q,q+1\},
  \qquad
  \#\{p:\Delta_p=q+1\}=r,
  \qquad
  \sum_{p=1}^m \Delta_p = 10,
\]
yields a D-optimal design via
\[
  j_p = \sum_{\ell=1}^p \Delta_\ell,
  \qquad p=1,\dots,m,
\]
so that the sensor nodes $j_p$ are (up to integer rounding) uniformly spaced
on $\{1,\dots,10\}$ and always include the free end $j_m=10$.
\end{theorem}

\begin{proof}
By Lemma~\ref{lem:min-det}, maximizing $\det K(S)$ over strictly increasing
integers $1\le j_1<\cdots<j_m\le 10$ is equivalent to maximizing
\[
  \prod_{p=1}^m \Delta_p
  \quad\text{subject to}\quad
  \Delta_p\in\mathbb{N},\ \Delta_p\ge 1,\ \sum_{p=1}^m \Delta_p = j_m \le 10.
\]

\paragraph{Step 1: saturation ($j_m=10$).}
If $j_m<10$, then increasing the last increment $\Delta_m\mapsto \Delta_m+1$
keeps feasibility and strictly increases the product, since
$\prod_{p=1}^m \Delta_p$ is multiplied by $(\Delta_m+1)/\Delta_m>1$.
Hence any maximizer must satisfy $j_m=10$, i.e.\ $\sum_{p=1}^m \Delta_p=10$.

\paragraph{Step 2: equal increments (discrete balancing).}
With $\sum_{p=1}^m \Delta_p=10$ fixed, the product is maximized when the
increments are as equal as possible.
Indeed, suppose there exist indices $a,b$ such that
$\Delta_a \ge \Delta_b + 2$.
Replacing $(\Delta_a,\Delta_b)$ by $(\Delta_a-1,\Delta_b+1)$ preserves the
sum and yields
\[
  (\Delta_a-1)(\Delta_b+1) - \Delta_a\Delta_b
  = \Delta_a - \Delta_b - 1 \ge 1,
\]
so the product strictly increases.
Hence any maximizer cannot contain such a pair.
Repeatedly applying this balancing step terminates at a configuration in which
all increments differ by at most $1$.
Writing
$
  q := \left\lfloor \frac{10}{m} \right\rfloor,
$
the condition that all increments differ by at most $1$ implies that each
$\Delta_p$ can take only one of the two values $q$ or $q+1$.
Let $r$ denote the number of indices $p\in\{1,\dots,m\}$ for which
$\Delta_p = q+1$.
Then the remaining $m-r$ increments equal $q$, and the sum constraint becomes
\[
  \sum_{p=1}^m \Delta_p
  = r(q+1) + (m-r)q
  = mq + r.
\]
Since $\sum_{p=1}^m \Delta_p = 10$ by construction, we must have
\[
  mq + r = 10,
  \qquad\text{hence}\qquad
  r = 10 - mq.
\]
Because $q=\lfloor 10/m\rfloor$, it follows that $0\le r < m$, so this choice of
$r$ is admissible and uniquely determined.
Consequently, the optimal increment sequence consists of exactly $r$ entries
equal to $q+1$ and $m-r$ entries equal to $q$, yielding the stated
characterization.
\end{proof}

\begin{remark}[Explicit sensor sets for small $m$]
\rm
For illustration, one convenient ``integer-uniform'' choice is to place the
larger increments $q+1$ as evenly as possible among the $m$ gaps.  This gives:
\begin{itemize}[leftmargin=2em]
\item $m=1$: $S^\star=\{10\}$;
\item $m=2$: $S^\star=\{5,10\}$ (increments $5,5$);
\item $m=3$: $S^\star=\{3,6,10\}$ (increments $3,3,4$) (equivalently $\{4,7,10\}$ with increments $4,3,3$);
\item $m=4$: $S^\star=\{2,5,7,10\}$ (increments $2,3,2,3$).
\end{itemize}
Other D-optimal sets are obtained by permuting the multiset of
increments (the determinant depends only on $\{\Delta_p\}_{p=1}^m$).
\end{remark}

\bibliographystyle{plain}
\bibliography{refs}

\end{document}